\documentclass{amsart}

\usepackage{amssymb}
\usepackage{amsmath}
\usepackage{amsthm}
\usepackage{amsfonts}

\usepackage{latexsym}

\usepackage[dvips]{epsfig}
\usepackage{amsbsy}
\usepackage{amsgen}
\usepackage{amscd}
\usepackage{amsopn}
\usepackage{amstext}
\usepackage{amsxtra}
\usepackage{xypic}

\newtheorem{theorem}{Theorem}[section]
\newtheorem{proposition}[theorem]{Proposition}
\newtheorem{lemma}[theorem]{Lemma}
\newtheorem{corollary}[theorem]{Corollary}

\theoremstyle{definition}
\newtheorem{definition}{Definition}

\begin{document}

\title[On some new model category structures from 
old, on the same underlying category]
{On some new model category structures from old, 
on the same underlying category}
\author[{A. E.} {Stanculescu}]{{Alexandru E.} {Stanculescu}}
\address{\newline \'{U}stav matematiky a statistiky
\newline Masarykova Univerzita,  Kotl\'{a}{\v{r}}sk{\'{a}} 2 \newline
611 37 Brno, Czech Republic}
\email{stanculescu@math.muni.cz}
\thanks{Supported by the project CZ.1.07/2.3.00/20.0003
of the Operational Programme Education for Competitiveness of 
the Ministry of Education, Youth and Sports of the Czech Republic
\newline
\indent}

\begin{abstract}
We make a study of $\ell\ell$-extensions of 
model category structures. We prove an 
existence result of $\ell\ell$-extensions, present 
some specific and some rather formal  results 
about them and give an application 
of the existence result to the homotopy 
theory of categories enriched over a 
monoidal model category.  
\end{abstract}

\maketitle
Given a category and a model structure 
on it, we inadequately say that an 
\emph{extension} of the model structure
is a model structure on the same category 
having more weak equivalences, and that an 
$\ell\ell$-\emph{extension} (or, \emph{extension of type} 
$\ell\ell$) of the model structure is an extension which has 
\emph{l}ess cofibrations and \emph{l}ess fibrations 
than the given one. We allow an $\ell\ell$-extension to have 
the same cofibrations or fibrations as the given model 
structure.

Every category having suitable limits and 
colimits admits a minimal model structure
in which the weak equivalences are the 
isomorphisms and all maps are cofibrations as 
well as fibrations. Any other model structure
on it is an $\ell\ell$-extension of the minimal model
structure. $\ell\ell$-extensions arise in disparate places 
such as the theory of $E^{2}$ model categories 
of Dwyer-Kan-Stover \cite{DKS} or the homotopy 
theory of (multi)categories enriched over monoidal 
model categories or of precategories enriched 
over cartesian model categories, for example.

In this paper we make a study of $\ell\ell$-extensions of 
model structures. Our interest in them comes, in part, 
from the homotopy theory of categories enriched 
over monoidal model categories. Our approach
to the study of $\ell\ell$-extensions is mainly influenced 
by Bousfield's work \cite{Bo}. His work is a vast 
generalization of \cite{DKS} and of others 
(see the Introduction to \cite{Bo}). The main 
result of the present work, Theorem 1.2, can
be seen as a generalization of \cite[Theorem 3.3]{Bo} 
tailored to capture a common feature of 
all of these homotopy theories.

The paper is organized as follows. 
In Section 1 we offer an existence result of 
$\ell\ell$-extensions (see Theorem 1.2). In Section 2 
we present some results about $\ell\ell$-extensions,
some of them which are in a specific context.
Sections 1 and 2 are independent of each other. 
In Section 3 we give an application of our 
existence result of $\ell\ell$-extensions to the 
homotopy theory of categories enriched over a 
monoidal model category (see Theorem 3.1).

\section{An existence result of $\ell\ell$-extensions}
Let $({\rm W,C,F})$ be a model structure on a category 
{\bf M}. {\rm W} stands for the class of weak equivalences,
{\rm C} for the class of cofibrations and {\rm F} for the
class of fibrations. Let ${\rm W}^{\mathcal{G}},
{\rm C}^{\mathcal{G}}$ and ${\rm F}^{\mathcal{G}}$ 
be three classes of maps of {\bf M} 
such that ${\rm W}\subset {\rm W}^{\mathcal{G}}$, 
${\rm C}^{\mathcal{G}}\subseteq {\rm C}$ and 
${\rm F}^{\mathcal{G}}\subseteq {\rm F}$.
In this section we give sufficient conditions for 
$({\rm W}^{\mathcal{G}},{\rm C}^{\mathcal{G}},
{\rm F}^{\mathcal{G}})$ to form a model structure on 
{\bf M} with ${\rm W}^{\mathcal{G}}$
as the class of weak equivalences,  ${\rm C}^{\mathcal{G}}$
as the class of cofibrations and ${\rm F}^{\mathcal{G}}$
as the class of fibrations. The result is stated as Theorem 1.2.

Our approach is heavily influenced by the proof 
of \cite[Theorem 3.3]{Bo}. However, there are 
differences. For example, it will follow from our 
result that everything after \cite[Proposition 3.17]{Bo} 
which pertains to the proof of \cite[Theorem 3.3]{Bo} 
is formal. This is somehow implicit in \cite{Bo}. We 
make it explicit in a way that uses less assumptions;
this difference will turn out to be essential for 
the application that we have in mind (see Section 3).
Our approach also encompasses the proof 
of \cite[Theorem 12.4]{Bo}. 

The next result sets the stage.
\begin{lemma} 
Let {\rm W, C} and {\rm F} be three classes of maps 
of a category {\bf M} with pushouts. We make the 
following assumptions.

$(1)$ {\rm W} has the two out of three property.

$(2)$ {\rm C} is closed under compositions and pushouts.

$(3)$ For every commutative solid arrow diagram in {\bf M}
\[
   \xymatrix{
A \ar[r] \ar[d]_{j} & X \ar[d]^{p}\\
B \ar[r] \ar@{-->} [ur] & Y\\
}
  \]
where $j$ is in ${\rm C}\cap {\rm W}$ and $p$ is in {\rm F},
there is a dotted arrow making everything commute.

$(4)$ Every map $f$ of {\bf M} factors as $f=qi$, where $i$ is a 
map in {\rm C} and $q$ is a map in ${\rm F}\cap {\rm W}$.

Then for every commutative solid arrow diagram in {\bf M}
\[
   \xymatrix{
A \ar[r] \ar[d]_{i} & X \ar[d]^{q}\\
B \ar[r] \ar@{-->} [ur] & Y\\
}
  \]
where $i$ is in {\rm C} and $q$ is in 
${\rm F}\cap {\rm W}$, there is a dotted 
arrow making everything commute.
\end{lemma}
\begin{proof}
For the first part, we construct a commutative 
diagram 
\[
\xymatrix{
A \ar[r]  \ar[d]_{i} & D \ar[r] \ar[d]^{j} & X \ar[d]^{q}\\
B \ar[r] & F \ar[r] & Y\\
}
   \]
with $j$ in ${\rm C}\cap {\rm W}$ and then apply (3)
to the right square diagram. Factor (4) the map 
$A\rightarrow X$ into a map $A\rightarrow D$
in {\rm C} followed by a map $D\rightarrow X$ in 
${\rm F}\cap {\rm W}$. Let $E$ be the pushout of 
$A\rightarrow D$ along $i$. By (2) the map 
$D\rightarrow E$ is in {\rm C}. Factor (4) the canonical 
map $E\rightarrow Y$ into a map $E\rightarrow F$ in {\rm C} 
followed by a map $F\rightarrow Y$ in ${\rm F}\cap {\rm W}$. 
Let $j$ be the composite $D\rightarrow E\rightarrow F$. 
$j$ is in {\rm C} by (2) and in {\rm W} by (1). 
\end{proof}
The initial object of a category, when it 
exists, is denoted by $\emptyset$. 
\begin{theorem} {\rm (after Bousfield)}
Let $({\rm W,C,F})$ be a model structure on a 
finitely bicomplete category {\bf M}.
Let ${\rm W}^{\mathcal{G}},{\rm C}^{\mathcal{G}}$ and
${\rm F}^{\mathcal{G}}$ be three classes of maps of 
{\bf M} such that ${\rm W}\subset {\rm W}^{\mathcal{G}}$, 
${\rm C}^{\mathcal{G}}\subseteq {\rm C}$ and 
${\rm F}^{\mathcal{G}}\subseteq {\rm F}$.
We make the following assumptions.

$(1)$ ${\rm W}^{\mathcal{G}}$ has the two out of three property.

$(2)$ ${\rm W}^{\mathcal{G}}$, ${\rm C}^{\mathcal{G}}$ 
and ${\rm F}^{\mathcal{G}}$ are closed under retracts.

$(3)$ ${\rm C}^{\mathcal{G}}$ is closed under compositions 
and pushouts.

$(4)$ For every object $X$ of {\bf M}, the map 
$\emptyset \rightarrow X$ is in ${\rm C}^{\mathcal{G}}$ 
if and only if $X$ is cofibrant.

$(5)$ Every cofibrant object of {\bf M} has an
${\bf M}^{\mathcal{G}}$-cylinder object. This
means that for every cofibrant object $X$ of {\bf M}
there is a factorization 
$$X\sqcup X\overset{i_{0}\sqcup i_{1}}\longrightarrow 
CylX\overset{p}\longrightarrow X$$
of the folding map $X\sqcup X\rightarrow X$
such that $i_{0}\sqcup i_{1}$ is in ${\rm C}^{\mathcal{G}}$
and $p$ is in ${\rm W}^{\mathcal{G}}$.

$(6)$ {\rm W} is closed under pushouts along maps 
from ${\rm C}^{\mathcal{G}}$ between cofibrant objects.

$(7)$ For every commutative solid arrow diagram in {\bf M}
\[
   \xymatrix{
A \ar[r] \ar[d]_{j} & X \ar[d]^{p}\\
B \ar[r] \ar@{-->} [ur] & Y\\
}
  \]
where $j$ is in ${\rm C}^{\mathcal{G}}
\cap {\rm W}^{\mathcal{G}}$ and $p$ is in 
${\rm F}^{\mathcal{G}}$, there is a dotted 
arrow making everything commute.

$(8)$ Every map $f$ of {\bf M} factors as $f=pj$, 
where $j$ is a map in ${\rm C}^{\mathcal{G}}
\cap {\rm W}^{\mathcal{G}}$ and $p$ is a map
in ${\rm F}^{\mathcal{G}}$.

Then the classes ${\rm W}^{\mathcal{G}},{\rm C}^{\mathcal{G}}$
and  ${\rm F}^{\mathcal{G}}$ form a model structure 
on {\bf M}.
\end{theorem}
\begin{proof}
Suppose we have shown that every map $f$ of {\bf M} 
factors as $f=qi$, where $i$ is a map in ${\rm C}^{\mathcal{G}}$ 
and $q$ is a map in ${\rm F}^{\mathcal{G}} \cap 
{\rm W}^{\mathcal{G}}$. Then, since we have (7), (8) and (3),
the fact that ${\rm W}^{\mathcal{G}},{\rm C}^{\mathcal{G}}$ 
and ${\rm F}^{\mathcal{G}}$ form a model structure 
would follow from Lemma 1.1.

To establish the desired factorization, let first $g:X\rightarrow Y$ 
be a map between cofibrant objects. We shall construct the 
mapping cylinder factorization $g=p_{g}i_{g}$ of $g$, where  
$i_{g}$ is in ${\rm C}^{\mathcal{G}}$ and $p_{g}$
in ${\rm W}^{\mathcal{G}}$. Let 
$X\sqcup X\overset{i_{0}\sqcup i_{1}}
\longrightarrow CylX\overset{p}\longrightarrow X$
be an ${\bf M}^{\mathcal{G}}$-cylinder object for $X$ (5).
Consider the following diagram, in which all 
squares are pushouts
\[
\xymatrix{
\emptyset \ar[r] \ar[d] & X \ar[d]^{i_{1}}\\
X \ar[r]^{i_{0}} \ar[d]_{g} & X\sqcup X \ar[r]^{i_{0}\sqcup i_{1}} 
\ar[d]^{X\sqcup g} & CylX \ar[d]^{\pi_{g}} \ar[r]^{p} & X\\
Y \ar[r]_{\sigma_{Y}} & X\sqcup Y \ar[r]_{h} & M_{f}\\
}
   \]
Put $i_{g}=h(X\sqcup g)i_{1}$ and $j_{g}=h\sigma_{Y}$.
There is a unique map $p_{g}:M_{f}\rightarrow Y$
such that $gp=p_{g}\pi_{g}$ and $p_{g}j_{g}=1_{Y}$.
Then $g=p_{g}i_{g}$. The map $i_{g}$ is in 
${\rm C}^{\mathcal{G}}$ by (3) and (4). The map $j_{g}$
is in ${\rm C}^{\mathcal{G}}\cap {\rm W}^{\mathcal{G}}$
using (1), (2), (7) and (8). Therefore $p_{g}$ is in 
${\rm W}^{\mathcal{G}}$ by (1).

Let now $f:X\rightarrow Y$ be an arbitrary map of {\bf M}. 
We can construct a commutative diagram
\[
   \xymatrix{
\tilde{X} \ar[r]^{u} \ar[d]_{\tilde{f}} & X \ar[d]^{f}\\
\tilde{Y} \ar[r]_{v} & Y\\
}
  \]
in which $u$ and $v$ are in {\rm W} and $\tilde{X}$ 
and $\tilde{Y}$ are cofibrant. By the above, $\tilde{f}$ 
factors as $\tilde{f}=p_{\tilde{f}}i_{\tilde{f}}$. Let $D$ 
be the pushout of $i_{\tilde{f}}:\tilde{X}\rightarrow 
M_{\tilde{f}}$ along $u$. By (3) the map $X\rightarrow D$ 
is in ${\rm C}^{\mathcal{G}}$. By (6) the map 
$M_{\tilde{f}}\rightarrow D$ is in {\rm W}.
Factor (8) the canonical map $D\rightarrow Y$ into
a map $D\rightarrow E$ in ${\rm C}^{\mathcal{G}}\cap 
{\rm W}^{\mathcal{G}}$ followed by map $q:E\rightarrow Y$
in ${\rm F}^{\mathcal{G}}$. $q$ is in 
${\rm W}^{\mathcal{G}}$ by (1). Take $i$ to be 
the composite $X\rightarrow D\rightarrow E$,
then the desired factorization is $f=qi$.
\end{proof}
We denote by ${\bf M}^{\mathcal{G}}$ the model 
structure constructed in Theorem 1.2. By construction, 
the cofibrant objects of ${\bf M}^{\mathcal{G}}$
coincide with the cofibrant objects of {\bf M}.
\begin{proposition}
The model category ${\bf M}^{\mathcal{G}}$ is 
left proper.
\end{proposition}
\begin{proof}
The proof proceeds exactly as in 
\cite[Proposition 3.27]{Bo}. For completeness 
we repeat it. Let 
\[
   \xymatrix{
A \ar[r]^{f} \ar[d]_{i} & X \ar[d]\\
B \ar[r] & Y\\
}
  \]
be a pushout diagram in {\bf M} in which $i$
is in ${\rm C}^{\mathcal{G}}$ and $f$ in
${\rm W}^{\mathcal{G}}$. As in \emph{loc. cit.}
we may assume that $A$ and $B$ are cofibrant.
Factor $f$ as a map $A\rightarrow \tilde{X}$ in {\rm C}
followed by a map $\tilde{X}\rightarrow Y$ in 
${\rm F}\cap{\rm W}$ and then take consecutive 
pushouts. The first map that factors $B\rightarrow Y$
is in ${\rm W}^{\mathcal{G}}$ as in \emph{loc. cit.},
the second map is in ${\rm W}^{\mathcal{G}}$ by 
assumption (6).
\end{proof}
To make the connection between Theorem 1.2 
and \cite[Theorems 3.3 or 12.4]{Bo}, let
$\mathcal{C}$ be a model category as 
in \emph{loc. cit.}. We take ${\bf M}=c\mathcal{C}$ 
with the Reedy model structure, ${\rm W}^{\mathcal{G}}$
to be the class of $\mathcal{G}$-equivalences,
${\rm C}^{\mathcal{G}}$ the class of 
$\mathcal{G}$-cofibrations and 
${\rm F}^{\mathcal{G}}$ the class of 
$\mathcal{G}$-fibrations.

Theorem 1.2 certainly admits variations.
We note one of them.
\begin{proposition}
Let $({\rm W,C,F})$ be a model structure on 
a finitely bicomplete category {\bf M}.
Let {\rm W'} and ${\rm W}^{\mathcal{G}}$
be two classes of maps of {\bf M} such that 
${\rm W}\subset {\rm W'} \subseteq 
{\rm W}^{\mathcal{G}}$. We define 
${\rm F}^{\mathcal{G}}$ to be the 
class of maps of {\bf M} having the right
lifting property with respect to the maps
in ${\rm C}\cap {\rm W'}$, and 
${\rm C}^{\mathcal{G}}$ to be the 
class of maps of {\bf M} having the left
lifting property with respect to the maps
in ${\rm F}^{\mathcal{G}}\cap 
{\rm W}^{\mathcal{G}}$.
We make the following assumptions.

$(1)$ {\rm W'} and ${\rm W}^{\mathcal{G}}$ 
have the two out of three property.

$(2)$ {\rm W'} and ${\rm W}^{\mathcal{G}}$ 
are closed under retracts.

$(3)$ Every map $f$ of {\bf M} factors as $f=pj$, 
where $j$ is a map in ${\rm C} \cap {\rm W'}$ 
and $p$ is a map in ${\rm F}^{\mathcal{G}}$.

$(4)$ Every map $f$ of {\bf M} factors as $f=qi$, 
where $i$ is a map in ${\rm C}^{\mathcal{G}}$
and $q$ is a map in ${\rm F}^{\mathcal{G}}
\cap {\rm W}^{\mathcal{G}}$.

Then the classes ${\rm W}^{\mathcal{G}},
{\rm C}^{\mathcal{G}}$ and  
${\rm F}^{\mathcal{G}}$ form a 
model structure on {\bf M}.
\end{proposition}
Theorem 1.2 admits a dual formulation.
For future reference we state it below.
The terminal object of a category, when it 
exists, is denoted by $\ast$.
\begin{theorem} 
Let $({\rm W,C,F})$ be a model structure on a 
finitely bicomplete category {\bf M}.
Let ${\rm W}^{\mathcal{G}},{\rm C}^{\mathcal{G}}$ 
and ${\rm F}^{\mathcal{G}}$ be three classes of maps of 
{\bf M} such that ${\rm W}\subset {\rm W}^{\mathcal{G}}$, 
${\rm C}^{\mathcal{G}}\subseteq {\rm C}$ and 
${\rm F}^{\mathcal{G}}\subseteq {\rm F}$.
We make the following assumptions.

$(1)$ ${\rm W}^{\mathcal{G}}$ has the two 
out of three property.

$(2)$ ${\rm W}^{\mathcal{G}}$, ${\rm C}^{\mathcal{G}}$ 
and ${\rm F}^{\mathcal{G}}$ are closed under retracts.

$(3)$ ${\rm F}^{\mathcal{G}}$ is closed under 
compositions and pullbacks.

$(4)$ For every object $X$ of {\bf M}, the map 
$X\rightarrow \ast$ is in ${\rm F}^{\mathcal{G}}$
if and only $X$ is fibrant.

$(5)$ Every fibrant object of {\bf M} has an
${\bf M}^{\mathcal{G}}$-path object. This
means that for every fibrant object $X$ of {\bf M}
there is a factorization 
$$X\overset{s}\rightarrow PathX 
\overset{p_{0}\times p_{1}}\longrightarrow X\times X$$
of the diagonal map $X\rightarrow X\times X$
such that $s$ is in ${\rm W}^{\mathcal{G}}$
and $p_{0}\times p_{1}$ is in ${\rm F}^{\mathcal{G}}$.

$(6)$ {\rm W} is closed under pullbacks along maps 
from ${\rm F}^{\mathcal{G}}$ between fibrant objects.

$(7)$ For every commutative solid arrow diagram in {\bf M}
\[
   \xymatrix{
A \ar[r] \ar[d]_{i} & X \ar[d]^{q}\\
B \ar[r] \ar@{-->} [ur] & Y\\
}
  \]
where $i$ is in ${\rm C}^{\mathcal{G}}$ and 
$q$ is in ${\rm F}^{\mathcal{G}}\cap 
{\rm W}^{\mathcal{G}}$, there is a dotted arrow 
making everything commute.

$(8)$ Every map $f$ of {\bf M} factors as $f=qi$, 
where $i$ is a map in ${\rm C}^{\mathcal{G}}$
and $q$ is a map in ${\rm F}^{\mathcal{G}}
\cap {\rm W}^{\mathcal{G}}$.

Then the classes ${\rm W}^{\mathcal{G}},
{\rm C}^{\mathcal{G}}$ and  
${\rm F}^{\mathcal{G}}$ form a right 
proper model structure on {\bf M}.
\end{theorem}
Motivated by the previous considerations and 
other naturally occuring examples we make the 
following
\begin{definition}
Let $({\rm W,C,F})$ be a model structure 
on a category {\bf M}. An $\ell\ell$-{\bf extension} (or,
{\bf extension of type} $\ell\ell$) of $({\rm W,C,F})$ 
is a model structure 
$({\rm W}^{\mathcal{G}},{\rm C}^{\mathcal{G}},
{\rm F}^{\mathcal{G}})$ on {\bf M} such that 
${\rm W}\subset {\rm W}^{\mathcal{G}}$, 
${\rm C}^{\mathcal{G}}\subseteq {\rm C}$ and 
${\rm F}^{\mathcal{G}}\subseteq {\rm F}$.
\end{definition}
Every model structure on a category 
is an $\ell\ell$-extension of its minimal 
model structure ({\rm W}=isomorphisms, 
{\rm C}=all maps, {\rm F}=all maps). Thus, 
Theorem 1.2 gives, in particular, a way to 
construct model  categories with all objects 
cofibrant.

An $\ell\ell$-extension as in Definition 1 for which 
${\rm C}^{\mathcal{G}}={\rm C}$ is sometimes
called left Bousfield localization, and one for which 
${\rm F}^{\mathcal{G}}={\rm F}$  is sometimes
called right Bousfield localization. Left and right
Bousfield localizations are ubiquitous \cite{Hi}.
\\

There are other kinds of extensions. For example,
given a category and a model structure on it, 
an $\ell m$-{\bf extension} of the given model structure
is another model structure on the same category
having more weak equivalences, \emph{l}ess
cofibrations and \emph{m}ore fibrations.
The following existence result of $\ell m$-extensions 
can be proved in a similar way as (the dual of) 
Theorem 1.2.
\begin{theorem}
Let $({\rm W,C,F})$ be a model structure on a 
finitely bicomplete category {\bf M}.
Let ${\rm W}^{\mathcal{G}},{\rm C}^{\mathcal{G}}$ 
and ${\rm F}^{\mathcal{G}}$ be three classes of maps of 
{\bf M} such that ${\rm W}\subset {\rm W}^{\mathcal{G}}$, 
${\rm C}^{\mathcal{G}}\subset {\rm C}$ and 
${\rm F} \subset{\rm F}^{\mathcal{G}}$.
We make the following assumptions.

$(1)$ ${\rm W}^{\mathcal{G}}$ has the two out of three property.

$(2)$ ${\rm W}^{\mathcal{G}}$, ${\rm C}^{\mathcal{G}}$ 
and ${\rm F}^{\mathcal{G}}$ are closed under retracts.

$(3)$ ${\rm F}^{\mathcal{G}}$ is closed under compositions 
and pullbacks.

$(4)$ For every commutative solid arrow diagram in {\bf M}
\[
   \xymatrix{
A \ar[r] \ar[d]_{i} & X \ar[d]^{q}\\
B \ar[r] \ar@{-->} [ur] & Y\\
}
  \]
where $i$ is in ${\rm C}^{\mathcal{G}}$ and 
$q$ is in ${\rm F}^{\mathcal{G}}\cap {\rm W}^{\mathcal{G}}$, 
there is a dotted arrow making everything commute.

$(5)$ Every map $f$ of {\bf M} factors as $f=qi$, 
where $i$ is a map in ${\rm C}^{\mathcal{G}}$ 
and $q$ is a map in ${\rm F}^{\mathcal{G}}\cap 
{\rm W}^{\mathcal{G}}$.

$(6)$ The model structure on {\bf M} is right proper.

Then the classes ${\rm W}^{\mathcal{G}},
{\rm C}^{\mathcal{G}}$ and  
${\rm F}^{\mathcal{G}}$ form a 
model structure on {\bf M}.
\end{theorem}

\section{On right derived functors and 
completions, and other formal results}
There seem to be few results about $\ell\ell$-extensions. 
In the first part of this section we make 
an attempt to view results like Theorem 6.2 and 
the first part of Theorem 6.5 from \cite{Bo}--or 
rather their generalizations, as explained in 
\cite[6.20 and 12.8]{Bo}, as results about (very 
specific) $\ell\ell$-extensions. We hope that our 
approach highlights both general and particular
aspects of this part of Bousfield's work. Inspired 
by \cite[Chapters 3,4,5 and 9]{Hi}, we study in the 
second part of this section the behaviour of some 
model category theoretical properties under the passage 
to an $\ell\ell$-extension.

Throughout, {\bf M} is a bicomplete category and 
$({\rm W}^{\mathcal{G}},{\rm C}^{\mathcal{G}},
{\rm F}^{\mathcal{G}})$ an $\ell\ell$-extension of a 
model structure $({\rm W,C,F})$ on {\bf M}. We denote by
${\bf M}^{\mathcal{G}}$ the $\ell\ell$-extension.
The fibrant (cofibrant) objects of ${\bf M}^{\mathcal{G}}$ 
will be referred to as $\mathcal{G}$-{\bf fibrant} 
($\mathcal{G}$-{\bf cofibrant}). The fibrant (cofibrant) 
objects with respect to the model structure 
$({\rm W,C,F})$ will be simply referred to as fibrant
(cofibrant).
\begin{proposition}
Suppose that the $\mathcal{G}$-cofibrant objects 
coincide with the cofibrant objects of {\bf M} and that
${\bf M}^{\mathcal{G}}$ is a simplicial model category. 
Let $\mathcal{N}$ be a simplicial category
and $T:{\bf M}\rightarrow \mathcal{N}$ a simplicial functor
with the property that $T$ sends the maps in {\rm W} 
between fibrant objects to isomorphisms. Let 
$H^{s}:\mathcal{N}\rightarrow \mathcal{N}'$ be a functor
which identifies strictly simplicially homotopic maps.
Then the composite $H^{s}T$ sends the maps in 
${\rm W}^{\mathcal{G}}$ between $\mathcal{G}$-fibrant 
objects to isomorphisms. 
\end{proposition}
\begin{proof}
Let $f:X\rightarrow Y$ be a map in ${\rm W}^{\mathcal{G}}$ 
between $\mathcal{G}$-fibrant objects. We can construct a 
commutative diagram
\[
   \xymatrix{
\tilde{X} \ar[r]^{u} \ar[d]_{\tilde{f}} & X \ar[d]^{f}\\
\tilde{Y} \ar[r]_{v} & Y\\
}
  \]
in which $u$ and $v$ are in ${\rm F}\cap {\rm W}$ 
and $\tilde{X}$ and $\tilde{Y}$ are cofibrant. Since
${\rm F}\cap {\rm W}\subseteq {\rm F}^{\mathcal{G}}\cap 
{\rm W}^{\mathcal{G}}$, it follows that $\tilde{X}$ and 
$\tilde{Y}$ are $\mathcal{G}$-fibrant. Thus, $T(\tilde{f})$
is a strict simplicial homotopy equivalence using 
\cite[Proposition 9.5.24(2)]{Hi}, and therefore
$H^{s}T(\tilde{f})$ is an isomorphism.
\end{proof}
Proposition 2.1 implies that the right derived functor
$\mathcal{R}_{\mathcal{G}}H^{s}T$ of $H^{s}T$ 
with respect to ${\bf M}^{\mathcal{G}}$ exists. 
We shall describe a (very particular) way 
to compute it.
\\

Let $\mathcal{G}'$ be a class of objects of {\bf M} 
which is invariant under {\rm W}. That is, if 
$X\rightarrow Y$ is in {\rm W}, then $X\in \mathcal{G}'$
if and only if $Y\in \mathcal{G}'$. We assume that
every $\mathcal{G}$-fibrant object is in $\mathcal{G}'$.
A {\bf weak} $\mathcal{G}$-{\bf fibrant approximation}
to an object $A$ of {\bf M} is a diagram
$A\overset{j}\rightarrow Y$, where $j$ is in 
${\rm W}^{\mathcal{G}}$ and  $Y$ is in $\mathcal{G}'$.

Let $cst:{\bf M}\rightarrow {\bf M}^{\Delta}$ be the
constant cosimplicial object functor. For an object 
$Y\in {\bf M}$, let $\vec{Y}$ be a Reedy fibrant 
approximation to $cstY$ in 
$({\bf M}^{\mathcal{G}})^{\Delta}$, and let
$\bar{\bar{Y}}=Tot\vec{Y}$. We have an induced 
map $\alpha:Y\rightarrow \bar{\bar{Y}}$.
\begin{lemma}
Suppose that the $\mathcal{G}$-cofibrant objects 
coincide with the cofibrant objects of {\bf M} and 
that ${\bf M}^{\mathcal{G}}$ is a simplicial model 
category. Let $\mathcal{N}$ be a simplicial category
and $T:{\bf M}\rightarrow \mathcal{N}$ a simplicial
functor with the property that $T$ sends the maps in 
{\rm W} to isomorphisms. Let $H^{s}:\mathcal{N}
\rightarrow \mathcal{N}'$ be a functor which identifies 
strictly simplicially homotopic maps. Suppose furthermore
that for each fibrant object $Y$ of {\bf M} which 
belongs to $\mathcal{G}'$, the map $\alpha$ is in 
${\rm W}^{\mathcal{G}}$ and the 
map $H^{s}T(\alpha)$ is an isomorphism.
Then $\mathcal{R}_{\mathcal{G}}H^{s}T$ can 
be computed using weak $\mathcal{G}$-fibrant 
approximations.
\end{lemma}
\begin{proof}
Let $A$ be an object of {\bf M}. Let $A\rightarrow 
\bar{A}$ be a map in ${\rm C}^{\mathcal{G}}
\cap {\rm W}^{\mathcal{G}}$ with $\bar{A}$
$\mathcal{G}$-fibrant and $A\rightarrow Y$ a weak 
$\mathcal{G}$-fibrant approximation to
$A$. Let $Y\rightarrow \underline{Y}$ be a 
fibrant approximation to $Y$. The diagram
\[
   \xymatrix{
A \ar[r] \ar[d] & Y \ar[r] & \underline{Y} 
\ar[r]^{\alpha} & \bar{\bar{\underline{Y}}}\\
\bar{A}\\
}
  \]
has a lifting $\bar{A}\rightarrow 
\bar{\bar{\underline{Y}}}$. By Proposition 
2.1 and assumptions it follows that 
$\mathcal{R}_{\mathcal{G}}H^{s}TA\cong
H^{s}TY$.
\end{proof}
To make the connection between Lemma 2.2
and \cite[Theorem 6.2]{Bo}, we take 
${\bf M}=c\mathcal{C}$ and $\mathcal{G}'$
to be the class of termwise $\mathcal{G}$-injective
objects \cite[Definition 6.1]{Bo}. Then a weak 
$\mathcal{G}$-fibrant approximation is just a weak 
$\mathcal{G}$-resolution as in \emph{loc. cit.}.
\begin{proposition}
Suppose that the $\mathcal{G}$-cofibrant objects 
coincide with the cofibrant objects of {\bf M} and that
${\bf M}^{\mathcal{G}}$ is a simplicial model category. 
Let $\mathcal{N}$ be a simplicial model category
and $T:{\bf M}\rightarrow \mathcal{N}$ a simplicial functor
with the property that $T$ sends the maps in {\rm W} 
between fibrant objects to weak equivalences. Then $T$ 
sends the maps in ${\rm W}^{\mathcal{G}}$ between 
$\mathcal{G}$-fibrant objects to weak equivalences. 
\end{proposition}
\begin{proof}
The proof is the same as for Proposition 2.1, using now
the fact that a simplicial functor between simplicial
model categories sends weak equivalences between
cofibrant-fibrant objects to weak equivalences.
\end{proof}
Proposition 2.3 implies that the total right derived 
functor $\mathcal{R}_{\mathcal{G}}T$ of $T$ 
with respect to ${\bf M}^{\mathcal{G}}$ exists. 
We shall describe a (very particular) way 
to compute it.
\begin{lemma}
Suppose that the $\mathcal{G}$-cofibrant objects 
coincide with the cofibrant objects of {\bf M} and 
that ${\bf M}^{\mathcal{G}}$ is a simplicial model 
category. Let $\mathcal{N}$ be a simplicial model 
category and $T:{\bf M}\rightarrow \mathcal{N}$ a 
simplicial functor with the property that $T$ sends the 
maps in {\rm W} between fibrant objects to weak 
equivalences. Suppose furthermore that for each 
fibrant object $Y$ of {\bf M} which belongs to 
$\mathcal{G}'$, the map $\alpha$ is in 
${\rm W}^{\mathcal{G}}$ and the 
map $T(\alpha)$ is a weak equivalence.
Let $A\rightarrow Y$ be a weak $\mathcal{G}$-fibrant 
approximation to an object $A$. Then 
$\mathcal{R}_{\mathcal{G}}TA\cong 
\mathcal{R}TY$, where $\mathcal{R}T$ is the 
total right derived functor of $T$.
\end{lemma}
\begin{proof}
The proof is the same as for Lemma 2.2, 
using Proposition 2.3.
\end{proof}
To make the connection between Lemma 2.4 
and the fist part of \cite[Theorem 6.5]{Bo},
we take ${\bf M}=c\mathcal{C}$, $\mathcal{G}'$
to be the class of termwise $\mathcal{G}$-injective
objects, $\mathcal{N}=\mathcal{C}$ and $T=Tot$.
\\

The next result can be seen as a generalization 
of \cite[Proposition 3.4.4]{Hi}.
\begin{proposition}
{\rm (1)} If {\bf M} is left proper and the cofibrant objects
coincide with the $\mathcal{G}$-cofibrant objects, 
then ${\bf M}^{\mathcal{G}}$ is left proper.

{\rm (2)} If {\bf M} is right proper and the fibrant 
objects coincide with the $\mathcal{G}$-fibrant objects, 
then ${\bf M}^{\mathcal{G}}$ is right proper.
\end{proposition}
\begin{proof}
The proof of (1) is similar to the proof of 
Proposition 1.3. The proof of (2) is dual.
\end{proof}
The next result can be seen as a generalization 
of \cite[Propositions 3.3.15 and 3.4.6]{Hi}.
To state it, we introduce some terminology.
Let {\bf M}$_{0}$ and {\bf M}$_{1}$ be two
full subcategories of {\bf M} with ${\bf M}_{0}
\subseteq {\bf M}_{1}$. We say that {\bf M}$_{0}$
is {\bf invariant in} {\bf M}$_{1}$ 
{\bf under} {\rm W} if for every commutative 
diagram
\[
   \xymatrix{
A \ar[r]^{u} \ar[d]_{f} & A' \ar[d]^{g}\\
B \ar[r]_{v} & B'\\
}
  \]
in which $u$ and $v$ are in {\rm W}
and $f$ and $g$ are in {\bf M}$_{1}$,
$f$ is in {\bf M}$_{0}$ if and only if 
$g$ is in {\bf M}$_{0}$.
\begin{proposition}
{\rm (1)} Let 
\[
   \xymatrix{
X \ar[rr]^{h} \ar[dr]_{f} & & Y \ar[dl]^{g}\\
& Z\\
}
  \]
be a commutative diagram in {\bf M} 
with $f\in {\rm F}$, $g\in {\rm F}^{\mathcal{G}}$
and $h\in {\rm W}$. Then $f\in {\rm F}^{\mathcal{G}}$.
If {\bf M} is left proper and ${\rm C}^{\mathcal{G}}$
is invariant in {\rm C} under {\rm W}, then the converse
holds, that is, $f\in {\rm F}^{\mathcal{G}}$ and 
$g\in {\rm F}$ imply $g\in {\rm F}^{\mathcal{G}}$.

{\rm (2)} Let 
\[
   \xymatrix{
& A \ar[dl]_{f} \ar[dr]^{g}\\
B \ar[rr]_{h}  & & C\\
}
  \]
be a commutative diagram in {\bf M} 
with $f\in {\rm C}^{\mathcal{G}}$, $g\in {\rm C}$,
and $h\in {\rm W}$. Then $g\in {\rm C}^{\mathcal{G}}$.
If {\bf M} is right proper and ${\rm F}^{\mathcal{G}}$
is invariant in {\rm F} under {\rm W}, then the converse
holds, that is, $g\in {\rm C}^{\mathcal{G}}$ and 
$f\in {\rm C}$ imply $f\in {\rm C}^{\mathcal{G}}$.
\end{proposition}
\begin{proof}
We will prove (1); the proof of (2) is dual. 
We prove the first part. Factor the map $h$ as 
a map $X\rightarrow E$ in {\rm C} followed by a 
map $q:E\rightarrow Y$ in ${\rm F}\cap {\rm W}$.
The diagram 
\[
   \xymatrix{
X \ar @{=} [rr] \ar[d] & & X \ar[d]^{f}\\
E \ar[r]_{q} & Y \ar[r]_{g} & Z\\
}
  \]
has a lifting, so $f$ is a retract of $gq$. 
But $gq\in {\rm F}^{\mathcal{G}}$.

We now prove the converse. We will show 
that every commutative diagram 
\[
   \xymatrix{
A \ar[r]^{t} \ar[d]_{i} & Y \ar[d]^{g}\\
B \ar[r]_{u} & Z\\
}
  \]
where $i$ is in ${\rm C}^{\mathcal{G}}\cap 
{\rm W}^{\mathcal{G}}$ has a lifting.
We can construct a commutative diagram
\[
   \xymatrix{
\tilde{A} \ar[r]^{q} \ar[d]_{\tilde{i}} & A \ar[d]^{i}\\
\tilde{B} \ar[r]_{r} & B\\
}
  \]
in which $q$ and $r$ are in {\rm W}, $\tilde{A}$ 
and  $\tilde{B}$ are cofibrant and $\tilde{i}$ 
is a cofibration. By assumption and 
\cite[Proposition 13.2.1(1)]{Hi} we may
assume without loss of generality that $i$ 
has cofibrant domain to begin with.
Then the proof proceeds exactly as in 
\cite[Proposition 3.4.6(1)]{Hi}.
\end{proof}
Let {\bf N} be another model category and
let $S:{\bf M}\rightleftarrows {\bf N}:T$
be a Quillen pair in which $S$ is the left 
adjoint. Let ${\bf N}^{\mathcal{G}}$ be an 
$\ell\ell$-extension of {\bf N}. The fibrant objects 
of ${\bf N}^{\mathcal{G}}$ will be referred to 
as $\mathcal{G}$-{\bf fibrant}. We assume
that the adjoint pair $(S,T)$ is also a Quillen 
pair with respect to the model structures 
${\bf M}^{\mathcal{G}}$ and 
${\bf N}^{\mathcal{G}}$. As such, we 
denote it by $(S^{\mathcal{G}},T^{\mathcal{G}})$.
\begin{proposition}
{\rm (1)} Suppose that the total right derived 
functor of $T$ is full and faithful. If the 
$\mathcal{G}$-cofibrant objects of {\bf M}
coincide with the cofibrant objects of 
{\bf M}, then the total right derived functor 
of $T^{\mathcal{G}}$ is full and faithful.

{\rm (2)} Suppose that the total left derived 
functor of $S$ is full and faithful. If the 
$\mathcal{G}$-fibrant objects of {\bf N}
coincide with the fibrant objects of {\bf N}, then 
the total left derived functor of $S^{\mathcal{G}}$ 
is full and faithful.
\end{proposition}
\begin{proof}
We will prove (1); the proof of (2) is dual. It is
sufficient to prove that for every $\mathcal{G}$-fibrant 
object $X$ of {\bf N} and for some cofibrant approximation 
$\widetilde{C}TX$ to $TX$ in ${\bf M}^{\mathcal{G}}$,
the composite map $S\widetilde{C}TX\rightarrow 
STX\rightarrow X$ is in ${\rm W}^{\mathcal{G}}$.
Let $\widetilde{C}TX$ be any cofibrant approximation
to $TX$ in {\bf M}. Since a $\mathcal{G}$-fibrant 
object is fibrant, the composite map 
$S\widetilde{C}TX\rightarrow STX\rightarrow X$ 
is in {\rm W} by hypothesis.
\end{proof}
\begin{corollary}
If $(S,T)$ is a Quillen equivalence, the 
$\mathcal{G}$-cofibrant objects of {\bf M}
coincide with the cofibrant objects of 
{\bf M} and the $\mathcal{G}$-fibrant objects 
of {\bf N} coincide with the fibrant objects of 
{\bf N}, then $(S^{\mathcal{G}},T^{\mathcal{G}})$
is a Quillen equivalence.
\end{corollary}
The next result can be seen as a generalization 
of the fact \cite[4.1.1(4)]{Hi} that ``a left Bousfield 
localization of a simplicial model category
is a simplicial model category''. Its proof
shows that \emph{loc. cit.} is actually a formal 
result. We recall from \cite[4.2.6 and 4.2.18]{Ho}
the notions of monoidal model category and 
$\mathcal{C}$-model category. In what follows 
we shall neglect the second part of these definitions.
\begin{proposition}
Suppose that {\bf M} is a 
$\mathcal{V}$-model category, for some
cofibrantly generated monoidal model 
category $\mathcal{V}$ which has a 
generating set of cofibrations with 
cofibrant domains. Let us write $X\ast K$
and $X^{K}$ for the tensor and cotensor
of $X\in {\bf M}$ with an object $K$ of 
$\mathcal{V}$. Then ${\bf M}^{\mathcal{G}}$ 
is a $\mathcal{V}$-model category (for the 
same tensor and cotensor) if and only if 

{\rm (1)} for every map $A\rightarrow B$ 
in ${\rm C}^{\mathcal{G}}$ and every generating
cofibration $J\rightarrow K$ of $\mathcal{V}$, 
the map $$A\ast K\underset{A\ast J}\coprod 
B\ast J\rightarrow B\ast K$$
is in ${\rm C}^{\mathcal{G}}$, and 

{\rm (2)} for every map $X\rightarrow Y$
in ${\rm F}^{\mathcal{G}}$ between
$\mathcal{G}$-fibrant objects and every 
object $K$ belonging to the set of domains 
and codomains of the generating cofibrations 
of $\mathcal{V}$, the map $X^{K}
\rightarrow Y^{K}$ is in ${\rm F}^{\mathcal{G}}$.

If {\bf M} is right proper,
{\rm (1)} can be replaced by

{\rm (1')} for every generating cofibration 
$J\rightarrow K$ of $\mathcal{V}$ and
every map $X\rightarrow Y$
in ${\rm F}\cap {\rm W}^{\mathcal{G}}$ 
between fibrant objects, the map
$$X^{K}\rightarrow X^{J}\underset{Y^{J}}
\prod Y^{K}$$
is in ${\rm F}^{\mathcal{G}}\cap 
{\rm W}^{\mathcal{G}}$.
\end{proposition}
\begin{proof}
For the equivalence between the (first part of the)
$\mathcal{V}$-model category axiom and (1)
and (2) one uses the fact that, in a model category, 
a cofibration is a weak equivalence if and only if it 
has the left lifting property with respect to every
fibration between fibrant objects \cite[Lemma 7.14]{JT}. 
For the rest one uses \cite[13.2.1(2)]{Hi}.
\end{proof}
\begin{lemma}
Suppose that {\bf M} is a 
$\mathcal{V}$-model category, for some
cofibrantly generated monoidal model 
category $\mathcal{V}$ which has a 
generating set of cofibrations with 
cofibrant domains. Let us write $X\ast K$
for the tensor of $X\in {\bf M}$ with an object 
$K$ of $\mathcal{V}$. Assume that 
${\bf M}^{\mathcal{G}}$ is a right 
Bousfield localization of {\bf M}.
Then ${\bf M}^{\mathcal{G}}$ 
is a $\mathcal{V}$-model category (for the 
same structure) if and only if for every 
$K$ belonging to the set of domains 
and codomains of the generating cofibrations 
of $\mathcal{V}$ and every $\mathcal{G}$-cofibrant
object $A$, $K\ast A$ is $\mathcal{G}$-cofibrant.
\end{lemma}
\begin{proof}
The proof is similar to the proof of Proposition 2.9.
\end{proof}

\section{Application: Categories enriched 
over monoidal model categories}
In this section we give the following application 
of Theorem 1.5. Let $\mathcal{V}$ be a closed category.
We denote by $\mathcal{V}$-{\bf Cat} the category 
whose objects are the small $\mathcal{V}$-categories 
and whose morphisms are the $\mathcal{V}$-functors.
When $\mathcal{V}$ is a model category satisfying
certain assumptions, $\mathcal{V}$-{\bf Cat} admits 
the fibred model structure \cite[4.4]{St}. Under further
assumptions on $\mathcal{V}$, we shall exhibit in
Theorem 3.1 an $\ell\ell$-extension of the fibred 
model structure.
\\

Let $\mathcal{V}$ be a monoidal model category 
with cofibrant unit $I$. Let {\bf Cat} be the category 
of small categories. We have a functor 
$[\_]_{\mathcal{V}}:\mathcal{V}\text{-}{\bf Cat} 
\rightarrow \mathbf{Cat}$ obtained by change of 
base along the symmetric monoidal composite functor
\[
\xymatrix{
\mathcal{V} \ar[r] & Ho(\mathcal{V})
\ar[rr]^{Ho(\mathcal{V})(I,\_)} & & Set\\
}
\]
Let $\mathcal{K}$ be a class of 
maps of $\mathcal{V}$. We say that 
a $\mathcal{V}$-functor $f:\mathcal{A}
\rightarrow \mathcal{B}$ is {\bf locally in} 
$\mathcal{K}$ if for each pair
$x,y$ of objects of $\mathcal{A}$, the 
map $f_{x,y}:\mathcal{A}(x,y)\rightarrow 
\mathcal{B}(fx,fy)$ is in $\mathcal{K}$.
\begin{definition} Let $f:\mathcal{A}\rightarrow 
\mathcal{B}$ be a morphism in $\mathcal{V}$-{\bf Cat}.

1. The morphism $f$ is a {\bf weak equivalence} 
if $f$ is locally a weak equivalence of $\mathcal{V}$ 
and $[f]_{\mathcal{V}}:[\mathcal{A}]_{\mathcal{V}}
\rightarrow [\mathcal{B}]_{\mathcal{V}}$
is essentially surjective.

2. The morphism $f$ is a {\bf fibration} if

$(a)$ $f$ is locally a fibration of $\mathcal{V}$, 
and

$(b)$ for any $x\in Ob(\mathcal{A})$, and any 
isomorphism $v:y'\rightarrow [f]_{\mathcal{V}}(x)$ 
in $[\mathcal{B}]_{\mathcal{V}}$, there exists an 
isomorphism $u:x'\rightarrow x$ in 
$[\mathcal{A}]_{\mathcal{V}}$ such 
that $[f]_{\mathcal{V}}(u)=v$.

3. The morphism $f$ is called a {\bf cofibration} 
if it has the left lifting property with respect to 
the fibrations which are weak equivalences.
\end{definition}
We denote by ${\rm W}^{\mathcal{G}}$
the class of weak equivalences, by 
${\rm C}^{\mathcal{G}}$ the class of 
cofibrations and by ${\rm F}^{\mathcal{G}}$
the class of fibrations. It follows directly from 
the definitions that a $\mathcal{V}$-functor is
in ${\rm F}^{\mathcal{G}}\cap {\rm W}^{\mathcal{G}}$
if and only if it is surjective on objects and locally a 
trivial fibration of $\mathcal{V}$.

Let now $\mathcal{V}$ be as in \cite[4.2]{St}.
Then the category $\mathcal{V}$-{\bf Cat}
admits a weak factorization system 
$({\rm C}^{\mathcal{G}},{\rm F}^{\mathcal{G}}
\cap {\rm W}^{\mathcal{G}})$ and the fibred
model structure \cite[4.4]{St}. A $\mathcal{V}$-category 
$\mathcal{A}$ is cofibrant (fibrant) in the fibred 
model structure if and only if $\mathcal{A}$ is
cofibrant (fibrant) in the model structure on the
category of $\mathcal{V}$-categories with fixed 
set of objects $Ob(\mathcal{A})$, and the map from 
the initial $\mathcal{V}$-category to $\mathcal{A}$ 
(from $\mathcal{A}$ to the terminal $\mathcal{V}$-category)
is in ${\rm C}^{\mathcal{G}}$ (${\rm F}^{\mathcal{G}}$) 
if and only if $\mathcal{A}$ is cofibrant (fibrant) 
in the fibred model structure.
\begin{theorem}
Suppose furthermore that $\mathcal{V}$ is right proper
and for every locally fibrant $\mathcal{V}$-category
$\mathcal{A}$, there is a factorization
$$\mathcal{A}\overset{s}\rightarrow Path
\mathcal{A}\overset{p_{0}\times p_{1}}
\longrightarrow \mathcal{A}\times \mathcal{A}$$
of the diagonal map $\mathcal{A}\rightarrow 
\mathcal{A}\times \mathcal{A}$
such that $s$ is in ${\rm W}^{\mathcal{G}}$
and $p_{0}\times p_{1}$ is in ${\rm F}^{\mathcal{G}}$.
Then the classes ${\rm W}^{\mathcal{G}},
{\rm C}^{\mathcal{G}}$ and  
${\rm F}^{\mathcal{G}}$ form a right 
proper model structure on
$\mathcal{V}$-{\bf Cat}. 
\end{theorem}
\begin{proof}
We take in Theorem 1.5 the category 
{\bf M} to be $\mathcal{V}$-{\bf Cat} 
regarded as having the fibred model 
structure. It is not difficult to show that 
assumptions (1), (2), (3) and (4) hold in our 
case. We have already shown above that
assumptions (7) and (8) also hold. 
Assumption (6) holds since 
$\mathcal{V}$ is right proper
whereas (5) is incorporated
in the statement of the Theorem.
\end{proof}
We said at the beginning of Section 1 that
Bousfield's proofs of \cite[Theorems 3.3 and 12.4]{Bo}
are, after a certain stage, formal. Hence the proofs
of the duals of \emph{loc. cit.} are, after a certain 
stage, formal. This raises the question whether 
one really needs our Theorem 1.5 for the proof 
of Theorem 3.1. We notice that the dual of 
\cite[Lemma 2.5]{Bo} does not hold in our 
example. Explicitly, if $\mathcal{A}\overset{f}
\rightarrow\mathcal{B}\overset{g}\rightarrow 
\mathcal{C}$ are $\mathcal{V}$-functors with 
$gf$ a fibration in the sense of Definition 2.2
and $g$ a fibration in the fibred model structure, 
then $g$ is not neccesarily a fibration.
\\

In order to apply Theorem 3.1 one needs to 
eventually construct the required factorization
of the diagonal. This is not immediate. A study of closed 
categories for which this factorization is possible is beyond 
the scope of this paper. Nevertheless, examples include 
the categories of: small groupoids, small categories, 
compactly generated Hausdorff spaces, chain
complexes of $R$-modules, simplicial $R$-modules 
(where $R$ is a commutative ring), small $2$-categories 
and small $\mathcal{V}$-categories (where $\mathcal{V}$
is a locally presentable closed category). Work
of B. van den Berg and R. Garner \cite{BG} seems
to suggest a construction of the required factorization 
for the category of simplicial sets.
\\

{\bf Acknowledgements.} I would like to thank the Referee
for his or her suggestions and I. Amrani, M. Korbel\'{a}{\v{r}},
G. Raptis and O. Raventos for useful discussions about 
the material of this paper.

\end{document}